\documentclass[10pt]{article}
\usepackage{amsfonts,color}
\usepackage{amsmath}
 \usepackage{epsfig}
\usepackage{lscape}
\usepackage{amssymb}
\usepackage{graphicx}
\usepackage{makeidx}
\usepackage{amssymb}
 \usepackage{footnote}
\usepackage[utf8]{inputenc}
\usepackage{amsmath}	
\usepackage{mathtools}
\usepackage{amsthm,amssymb,wasysym}
\usepackage{amsfonts}
\usepackage[english]{babel}

\makeatletter
\let\@fnsymbol\@arabic
\makeatother

\definecolor{purple}{rgb}{.5,.1,.7}

\allowdisplaybreaks[1]

 \newtheorem{theorem}{Theorem}[section]
  
 \newtheorem{lemma}[theorem]{Lemma}
 
 \newtheorem{proposition}[theorem]{Proposition}
 \newtheorem{cor}[theorem]{Corollary}

\newcommand{\R}{\mathbb{R}}
 \newcommand{\N}{\mathbb{N}}
 \newcommand{\lam}{\lambda}

 \newcommand{\tel}[1]{\frac{1}{#1}}

 \DeclareMathOperator{\diag}{diag}

 \DeclareMathOperator{\dev}{dev}

\def\barr{\begin{array}}
\def\id{1\!\!1}
\def\tr{\textrm{tr}}

\def\dd{\displaystyle}

\def\barr{\begin{array}}
\def\earr{\end{array}}
\def\bec#1{\begin{equation}\label{#1}}
\def\becn{\begin{equation*}}
\def\endec{\end{equation}}
\def\endecn{\end{equation*}}

\begin{document}
\title{The exponentiated Hencky-logarithmic strain energy. Improvement of  planar polyconvexity}
\author{
Ionel-Dumitrel Ghiba\thanks{Ionel-Dumitrel Ghiba, \ \ \ \ Lehrstuhl f\"{u}r Nichtlineare Analysis und Modellierung, Fakult\"{a}t f\"{u}r Mathematik, Universit\"{a}t Duisburg-Essen, Thea-Leymann Str. 9, 45127 Essen, Germany;  Alexandru Ioan Cuza University of Ia\c si, Department of Mathematics,  Blvd. Carol I, no. 11, 700506 Ia\c si,
Romania; and  Octav Mayer Institute of Mathematics of the
Romanian Academy, Ia\c si Branch,  700505 Ia\c si, email: dumitrel.ghiba@uni-due.de, dumitrel.ghiba@uaic.ro}
  \quad
and \quad
Patrizio Neff\,\thanks{Corresponding author: Patrizio Neff,  \ \ Head of Lehrstuhl f\"{u}r Nichtlineare Analysis und Modellierung, Fakult\"{a}t f\"{u}r Mathematik, Universit\"{a}t Duisburg-Essen,  Thea-Leymann Str. 9, 45127 Essen, Germany, email: patrizio.neff@uni-due.de}  \quad
and \quad
Miroslav {\v{S}}ilhav{\'y}\thanks{Mathematical Institute AS CR, \v{Z}itn\'{a}25, 115 67 Praha 1,
Czech Republic, e-mail: silhavy@math.cas.cz}}
\maketitle

\begin{abstract}

In this paper we improve the result about the polyconvexity of the energies from the  family of isotropic volumetric-isochoric decoupled strain exponentiated Hencky energies defined in the first part of this series, i.e. \begin{align*}\hspace{-2mm}
 W_{_{\rm eH}}(F)= \left\{\begin{array}{lll}
\dd\frac{\mu}{k}\,e^{k\,\|{\rm dev}_n\log U\|^2}+\frac{\kappa}{2\,\widehat{k}}\,e^{\widehat{k}\,[(\log \det U)]^2}&\text{if}& \det\, F>0,\vspace{2mm}\\
+\infty &\text{if} &\det F\leq 0\,,
\end{array}\right.
\end{align*}  where $F=\nabla \varphi$ is the gradient of deformation,  $U=\sqrt{F^T F}$ is the right stretch tensor and $\dev_n\log {U}$
 is the deviatoric part  of the strain tensor $\log U$. The main result in this paper is that in plane elastostatics, i.e. for $n=2$,  the energies of this family are polyconvex for $k\geq \frac{1}{4}$, $\widehat{k}\geq \frac{1}{8}$, extending a previous result which proves  polyconvexity for $k\geq \frac{1}{3}$,  $\widehat{k}\geq \frac{1}{8}$.   This  leads immediately  to an extension of the  existence result.
\\
\vspace*{0.25cm}
\end{abstract}

\section{Introduction}
\subsection{Exponentiated Hencky energy}

 In a previous  series of papers \cite{NeffGhibaLankeit,NeffGhibaPoly,NeffGhibaPlasticity} we have modified the Hencky energy and  considered the family of energies
\begin{align}\label{thdefHen}\hspace{-2mm}
 W_{_{\rm eH}}(F)=W_{_{\rm eH}}^{\text{\rm iso}}(\frac F{\det F^{\frac{1}{2}}})+W_{_{\rm eH}}^{\text{\rm vol}}(\det F^{\tel 2}\cdot \id) = \left\{\begin{array}{lll}
\dd\frac{\mu}{k}\,e^{k\,\|{\rm dev}_2\log U\|^2}+\frac{\kappa}{2\,\widehat{k}}\,e^{\widehat{k}\,[(\log \det U)]^2}&\text{if}& \det\, F>0,\vspace{2mm}\\
+\infty &\text{if} &\det F\leq 0\,.
\end{array}\right.
\end{align}
 We have called this the exponential Hencky energy.
Here, $\mu>0$ is the infinitesimal shear modulus,
$\kappa=\frac{2\mu+3\lambda}{3}>0$ is the infinitesimal bulk modulus with $\lambda$ the first Lam\'{e} constant, $k,\widehat{k}$ are additional dimensionless
parameters, $F=\nabla \varphi$ is the gradient of deformation,  $U=\sqrt{F^T F}$ is the right stretch tensor and $\dev_2\log {U} =\log {U}-\frac{1}{2}\,
\tr(\log {U})\cdot\id$
 is the deviatoric part  of the strain tensor $\log U$. For $X\in\R^{2\times 2}$,
  $\|{X}\|^2=\langle {X},{X}\rangle$ is the Frobenius tensor norm, $\tr{(X)}=\langle {X},{\id}\rangle$ and $\id$ denotes the identity tensor on $\R^{2\times 2}$. For further notations  we refer to \cite{NeffGhibaLankeit}.

 Our renewed interest in the Hencky energy is motivated by a recent finding that the Hencky energy (not the logarithmic strain itself) exhibits  a fundamental property. By purely differential geometric reasoning, in  forthcoming papers \cite{NeffEidelOsterbrinkMartin_Riemannianapproach,Neff_Osterbrink_Martin_hencky13,Neff_Nagatsukasa_logpolar13} (see also  {\cite{Neff_log_inequality13,LankeitNeffNakatsukasa}})  it will be shown that
\begin{align}\label{geoprop}
{\rm dist}^2_{{\rm geod}}\left((\det F)^{1/n}\cdot \id, {\rm SO}(n)\right)&={\rm dist}^2_{{\rm geod,\mathbb{R}_+\cdot \id}}\left((\det F)^{1/n}\cdot \id, \id\right)=|\log \det F|^2,\notag\\
{\rm dist}^2_{{\rm geod}}\left( \frac{F}{(\det F)^{1/n}}, {\rm SO}(n)\right)&={\rm dist}^2_{{\rm geod,{\rm SL}(n)}}\left( \frac{F}{(\det F)^{1/n}}, {\rm SO}(n)\right)=\|\dev_n \log U\|^2,
\end{align}
where ${\rm dist}_{{\rm geod}}$ is the canonical left invariant geodesic distance on the Lie group ${\rm GL}^+(n)$ and ${\rm dist}_{{\rm geod,{\rm SL}(n)}}$, ${\rm dist}_{{\rm geod,\mathbb{R}_+\cdot \id}}$ denote the corresponding geodesic distances on the Lie groups ${\rm SL}(n)$ and $\mathbb{R}_+\cdot \id$, respectively (see \cite{Neff_Osterbrink_Martin_hencky13,Neff_Nagatsukasa_logpolar13}). On the other hand,   some constitutive issues, e.g.  the invertible true-stress-true-strain relation and the monotonicity of the Cauchy stress tensor as a function of $\qopname \relax o{log}B$,  where $B=F\, F^T$ is the left Cauchy-Green tensor,  recommend the energies from our family of exponentiated energies as good candidates in the study of nonlinear deformations. {Moreover, in the first part \cite{NeffGhibaLankeit} it is shown that the proposed energies have some other very useful properties: analytical solutions are in  agreement with Bell's experimental data; planar pure Cauchy shear stress produces biaxial pure shear strain and the value 0.5 of Poisson's ratio corresponds to exact incompressibility. It is found \cite{NeffGhibaLankeit} that the analytical expression of the pressure is in concordance with the classical Bridgman’s compression data for natural rubber. An immediate application to rubber-like materials is proposed in \cite{Montella15}. }
We have also shown that the energies from the family of exponentiated energies  improve several features of the formulation with respect to mathematical issues regarding well-posedness. We have established that, in planar elasto-statics, the exponentiated energies $W_{_{\rm eH}}$    satisfy the Legendre-Hadamard condition (rank-one convexity) \cite{NeffGhibaLankeit}   for $k\geq \frac{1}{4}$, $\widehat{k}\geq \frac{1}{8}$, while they are  polyconvex \cite{NeffGhibaPoly}  for $k\geq \frac{1}{3}$, $\widehat{k}\geq \frac{1}{8}$ and satisfy a coercivity  estimate \cite{NeffGhibaPoly} for $k> 0$, $\widehat{k}> 0$. These results  now allow us to show the existence of minimizers  \cite{NeffGhibaPoly} for $k\geq \frac{1}{3}$, $\widehat{k}\geq \frac{1}{8}$.

\subsection{Polyconvexity}

  The notion of polyconvexity has been introduced into the framework of
elasticity by  John Ball in his seminal paper \cite{Ball77,Ball78,Raoult86}. Various nonlinear issues, results and extensive references are collected in Dacorogna \cite{Dacorogna08}. In the two dimensional case, a free energy function $W(F)$ is called  polyconvex if and only if it is expressible in the form
$W(F) =P(F,\det F)$, $P:\mathbb{R}^{10}\rightarrow\mathbb{R}$, where $P(\cdot,\cdot)$ is convex. Polyconvexity is the cornerstone notion for a proof of the existence of minimizers by the direct methods of the calculus of variations
for energy functions satisfying no polynomial growth conditions, which is  the case in nonlinear elasticity since one has the natural requirement
$W(F)\rightarrow\infty$ as $\det F\rightarrow0$. Polyconvexity is best understood for isotropic energy functions, but it is  not restricted to isotropic response.
The polyconvexity condition in the case of space dimension 2 was conclusively discussed  by Rosakis \cite{Rosakis98} and \v{S}ilhav\'{y} \cite{Silhavy97,Silhavy99b,Silhavy02bb,Silhavy03,silhavy2002monotonicity,vsilhavy2001rank,SilhavyPRE99}, while the case of arbitrary spatial dimension was studied by Mielke \cite{Mielke05JC}.  The $n$-dimensional case of  the theorem established by Ball \cite[page 367]{Ball77} has been reconsidered by Dacorogna and Marcellini \cite{DacorognaMarcellini}, Dacorogna and Koshigoe \cite{DacorognaKoshigoe} and  Dacorogna and Marechal \cite{DacorognaMarechal}. It was a long standing open question how to extend the notion of polyconvexity in a meaningful way to anisotropic materials \cite{Ball02}. An answer has been provided in a series  of papers \cite{Schroeder_Neff_Ebbing07,cism_book_schroeder_neff09,Schroeder_Neff01,Hartmann_Neff02,
Schroeder_Neff04,Balzani_Neff_Schroeder05,Schroeder_Neff_Ebbing07}.

Rank-one convexity domains for the Hencky energy \[
	W_{_{\rm H}}(F) = \widehat{W}_{_{\rm H}}(U) \;=\; {\mu}\,\|{\rm dev}_n\log U\|^2+\frac{\kappa}{2}\,[{\rm tr}(\log U)]^2\,
\]
have been established in  \cite{xiao1997logarithmic,Neff_Diss00,Bertram05,glugegraphical}. Satisfaction of the Baker-Ericksen inequalities in terms of the logarithmic strain tensor is discussed in  \cite{Criscione2000,Criscione2005}, while necessary conditions for Legendre-Hadamard ellipticity are given in  \cite{Walton05}.
\subsection{Motivation}
We have remarked that the rank-one convexity holds true for $k\geq \frac{1}{4}$, while the polyconvexity holds true for $k\geq \frac{1}{3}$. Hence, the following question arose: is there a  gap for $\frac{1}{3}> k\geq \frac{1}{4}$?  In a previous work we have used the sufficiency condition for polyconvexity which  has been discovered by Steigmann \cite{SteigmannMMS03,SteigmannQJ03}. Eventually, it is based on a polyconvexity criterion of Ball \cite{Ball77}, but it allows one to express polyconvexity directly in terms of the principal isotropic invariants of the right stretch tensor $U$ (see also \cite{Davis57,Ogden83,Lewis03,Lewis96,Lewis96b,Borwein2010,eremeyev2007constitutive}). As  it turns out, in plane elastostatics, Steigmann's criterion is already
hidden in another sufficiency criterion for polyconvexity given earlier by Rosakis \cite{Rosakis92}. However,  Steigmann's criterion is clearly not necessary for polyconvexity.  Hence, our previous results may be improved.

In this paper, we use a direct approach based on the fact that the function
$Y:[1,\infty)\rightarrow\mathbb{R}$  given by
$
Y(\theta)=e^{\frac{k}{2}\, \log^2\theta}, \ \theta\in [1,\infty),
$
is convex for $k\geq \frac{1}{4}$ and  it is also increasing for $\theta\geq 1$, while $Z:{\rm GL}^+(2)\rightarrow[1,\infty)$  given by
$
Z(F)=\frac{\lambda_{\rm max}^2}{\det F}, \ F \in {\rm GL}^+(2)$, where $\lambda_{\rm max}$ is the largest singular values (principal stretches)  of $F$,  is polyconvex. Therefore,  we prove that   the exponentiated Hencky energies are in fact polyconvex  for $k\geq \frac{1}{4}$ and the main existence results is also valid for these values of the fitting parameter $k$.

\section{Preliminaries. Auxiliary results}

\subsection{Formulation of the static problem in the planar case}\label{Formulation}\setcounter{equation}{0}
The static problem in the planar case consists in finding the solution $\varphi$ of the equilibrium  equation
\begin{align}\label{exst}
0={\rm Div} \,S_1(\nabla \varphi)\qquad  \text{in} \qquad \Omega\subset
{\R^2},
\end{align}
where the first Piola-Kirchhoff stress tensor corresponding to the energy $W_{_{\rm eH}}(F)$ is given by the constitutive equation
\begin{align}\label{pks}
&S_1(F)=\left[2{\mu}\,e^{k\,\|\dev_2\log\,U\|^2}\cdot \dev_2\log\,U+{\kappa}\,e^{\widehat{k}\,[\tr(\log U)]^2}\,\tr(\log U)\cdot \id\right]F^{-T},  \qquad x\in\overline{\Omega},
\end{align}
with $F=\nabla \varphi, \ U= \sqrt{F^TF}$. The above system of equations is supplemented, in the
case of the mixed problem,  by the boundary conditions
\begin{align}\label{cl1}
{\varphi}({x})&=\widehat{\varphi}_i({x}) \qquad \text{ on  }\quad \Gamma_D,\qquad
{S}_1({x}).\,n=\widehat{s}_1({x}) \qquad \text{ on  }\quad \Gamma_N,
\end{align}
where $\Gamma_D,\Gamma_N$  are subsets of the boundary $\partial \Omega$, so that $\Gamma_D\cup\overline{\Gamma}_N=\partial \Omega$,
$\Gamma_D\cap{\Gamma}_N=\emptyset$, ${n}$ is the unit outward normal to the boundary and  $\widehat{\varphi}_i, \widehat{s}_1$ are prescribed fields.

\subsection{Auxiliary results}
In \cite{NeffGhibaLankeit} and \cite{NeffGhibaPoly} the following results were established.
\begin{lemma} {\rm \cite{NeffGhibaLankeit}}
\label{thm:w2d}
 Let $k\in\R$ and the matrix $F\in{\rm GL}^+(2)$ with singular values $\lam_1,\lam_2$.  Then
 \begin{align}\label{wf3}
W_{_{\rm eH}}^{\rm iso}(F)=e^{k\,\|{\rm dev}_2 \log U\|^2}=e^{k\,\|\log \frac{U}{\det U^{1/2}}\|^2}=g(\lambda_1,\lambda_2),\ \ \text{where}\  g:\mathbb{R}^2_+\rightarrow\mathbb{R},\ \  g(\lambda_1,\lambda_2):=e^{\frac{k}{2}\,\left(\log \frac{\lambda_1}{\lambda_2}\right)^2}.
\end{align}
\end{lemma}

\begin{lemma}\label{lemaJH} {\rm \cite{NeffGhibaLankeit}} Let $m\in\N$. Then the
 function
$
 t\mapsto e^{\widehat{k}\,(\log(t))^m}
$
is convex if and only if
$
 \widehat{k}\geq \tel{m^{(m+1)}}.
$
\end{lemma}

This lemma together with a results established in  \cite[page 213]{Dacorogna08} led to:

\begin{proposition}\label{corolarbun}{\rm (Convexity of the volumetric part) \cite{NeffGhibaLankeit}}
The function
\[
 F\mapsto W_{_{\rm eH}}^{\rm vol}(F):=e^{\widehat{k}\,(\log\det F)^m}, \quad F\in {\rm GL}^+(n)
\]
is rank-one convex in $F$ for $\widehat{k}\geq \tel{m^{(m+1)}}$. ({More explicitly, for $m=2$ this means $\widehat{k}\geq \tel8$, in case of $m=3$ rank-one convexity holds for $\widehat{k}\geq \tel{81}$.})
\end{proposition}

\begin{theorem}{\rm (Coercivity) \cite{NeffGhibaPoly}}  Assume for the elastic moduli $\mu>0,\ \kappa>0$ and $k>0, \ \widehat{k}>0$.  Consider  the energy
$
I(\varphi)=\int_\Omega W_{_{\rm eH}}(\nabla \varphi(x)) \,dx,
$
 where
$
W_{_{\rm eH}}(F)=\widehat{W}_{_{\rm eH}}(U)= \frac{\mu}{k}\,e^{k\,\|\dev_2\log U\|^2}+\frac{\kappa}{2\,\widehat{k}}\,e^{\widehat{k}\,|{\rm tr}(\log U)|^2}.
$ Then $I(\varphi)$ is {\bf  $q$-coercive} for any $1\leq q<\infty$.
\end{theorem}

\section{Improved polyconvexity result }\setcounter{equation}{0}

By Lemma \ref{wf3}, the function $W_{_{\rm eH}}^{\rm iso}(F)=e^{k\,\|{\rm dev}_2 \log U\|^2}$ is given by
$
W_{_{\rm eH}}^{\rm iso}(F)=e^{\frac{k}{2}\,\log ^2\frac{\lambda_1}{\lambda_2}}=e^{\frac{k}{2}\,\log ^2\frac{\lambda_1^2}{\det F}}
$
for each $F\in{\rm GL}^+(2)$, where $\lambda_1\geq \lambda_2$ is  an ordered pair of singular values of $F$. We view $W_{_{\rm eH}}^{\rm iso}$ as the composition $W_{_{\rm eH}}^{\rm iso}=Y\circ Z$, where $Y:[1,\infty)\rightarrow\mathbb{R}$ is given by
$
Y(\theta)=e^{\frac{k}{2}\, \log^2\theta},  \theta\in [1,\infty),
$
and $Z:{\rm GL}^+(2)\rightarrow[1,\infty)$ is given by
$
Z(F)=\frac{\lambda_1^2}{\det F}\,, \ F \in {\rm GL}^+(2).
$
For $k\geq \frac{1}{4}$ the function $Y(\theta)$ is convex  and  it is also increasing, see Lemma \ref{lemaJH} and Figure \ref{grafy} .
\begin{figure}[h!]\begin{center}
\begin{minipage}[h]{0.7\linewidth}
\includegraphics[scale=0.8]{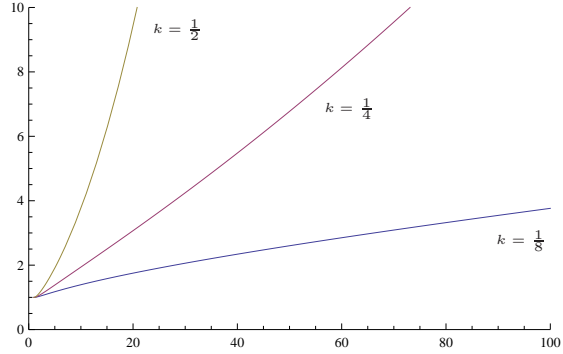}
\centering
\put(-25,40){\tiny{$k=\frac{1}{8}$}}
\put(-90,90){\tiny{$k=\frac{1}{4}$}}
\put(-155,120){\tiny{$k=\frac{1}{2}$}}
\caption{\footnotesize{ The graphic of the function $Y:[1,\infty)\rightarrow\mathbb{R}$,
$
Y(\theta)=e^{\frac{k}{2}\, \log^2\theta},  \theta\in [1,\infty),
$ for different values of the fitting parameter $k$. }}
\label{grafy}
\end{minipage}
\end{center}
\end{figure}
Hence, in order to prove the polyconvexity of $W_{_{\rm eH}}^{\rm iso}$, in view of Lemma \ref{lemaJH}, it is suffices to prove that $Z$ is polyconvex, since we have the following result:
\begin{lemma}
If $\, Y:[1,\infty)\rightarrow\mathbb{R}$ is increasing and convex and $Z:{\rm GL}^+(2)\rightarrow[1,\infty)$  is polyconvex, then $Y\circ Z$ is polyconvex.
\end{lemma}
\begin{proof}
This is proved in exactly the same way as \cite[Theorem 5.1]{Rockafellar70}; only the elements $x$ and $y$ of the proof have to be interpreted as pairs $(F, \delta)$, where $F\in{\rm GL}^+(2)$ and $\delta>0$. The generalization to arbitrary dimension, if needed, is straightforward.
\end{proof}
\begin{lemma}
The function $Z$ is polyconvex on ${\rm GL}^+(2)$.
\end{lemma}
\begin{proof}
An elementary evaluation of the Hessian matrix shows that for $p\geq 0, q\geq 0$ the function $f(a,b)=\frac{a^p}{b^q}$ is convex on $[0,\infty)\times [0,\infty)$ if and only if $q\leq p-1$. We take $p=2$ and $q=1$ and we observe that
\begin{align}
Z(F)=f(\lambda_1,\det F),
\end{align}
for every $F\in {\rm GL}^+(2)$. Let
$
f_1(F,\delta)=f(\lambda_1,\delta),
 $
for each $F\in  {\rm GL}^+(2)$ and $\delta>0$. Using the well-known fact that $F\mapsto \lambda_1(F)$ (the largest singular value) is convex (see e.g. \cite[Lemma 5.3]{Ball77} and also Appendix \ref{convmax}), we now prove that $f_1$ is convex. Indeed, since $f$ is increasing in the first argument, we have for any $t$, $0\leq t\leq 1$, and any $F_
\alpha\in  {\rm GL}^+(2)$, $\alpha=1,2$, that
\begin{align}
f_1((1-t)\, F_1+t\, F_2, (1-t)\, \delta_1+t\, \delta_2)&=
f(\lambda_1((1-t)\, F_1+t\, F_2), (1-t)\, \delta_1+t\, \delta_2)\\
&\leq (1-t)\, f(\lambda_1( F_1), \delta_1)+t\,f(\lambda_1(F_2), \delta_2)=(1-t)\, f_1( F_1, \delta_1)+t\,f_1(F_2, \delta_2),\notag
\end{align}
where we used the convexity of $f$ asserted above. Thus
$
Z(F)=f_1(F,\det F)
$
is polyconvex.
\end{proof}
Therefore, we conclude:
\begin{proposition}
If $k\geq \frac{1}{4}$, then the function $F\mapsto {W_{_{\rm eH}}^{\rm iso}}(F)=e^{k\, \|\dev_2\log U\|^2}$ is polyconvex on ${\rm GL}^+(2)$.
\end{proposition}

In view of the above proposition and using Corollary \ref{corolarbun}  we conclude that:

\begin{theorem}\label{mainth} The functions $W_{_{\rm eH}}:\R^{2\times 2}\to \R$ from the family of exponentiated Hencky type energies
are polyconvex for $\mu>0, \kappa>0$, $k\geq\dd\frac{1}{4}$ and $\widehat{k}\dd\geq \tel8$.
\end{theorem}

\section{Existence result}

In plane elastostatics, having proved the coercivity and the  polyconvexity of the energy $W_{_{\rm eH}}(U)$ for $k\geq \frac{1}{4}$ and $\widehat{k}\geq\frac{1}{8}$, it is a standard matter to improve the existence
results established in \cite{NeffGhibaPoly}.

\begin{theorem}\label{mainexist}{\rm (Existence of minimizers)} Let the reference configuration
$\Omega\subset \mathbb{R}^2$ be a bounded smooth domain and let $\Gamma_D$ be a non-empty and relatively open part of the boundary  $\partial \Omega$. Assume that
$I(\varphi)=\int_\Omega W_{_{\rm eH}}(\nabla \varphi(x)) dx$ where
$
W_{_{\rm eH}}(F)=\widehat{W}_{_{\rm eH}}(U)= \frac{\mu}{k}\,e^{k\,\|\dev_2\log U\|^2}+\frac{\kappa}{2\widehat{k}}\,e^{\widehat{k}\,|{\rm tr}(\log U)|^2}.
$
Let $\varphi_0\in W^{1,q}(\Omega), \ q\geq1$ be given with $I(\varphi_0)<\infty$ and  $\mu>0,\ \kappa>0$, $k\geq \frac{1}{4}$ and $\widehat{k}\geq \frac{1}{8}$. Then the problem
\begin{align}
\min\left\{I(\varphi)=\int_\Omega W_{_{\rm eH}}(\nabla \varphi(x)) dx, \ \varphi(x)=\varphi_0(x) \quad \text{for} \ x\in \Gamma_D\subset \partial \Omega,\quad  \varphi\in W^{1,q}(\Omega)\right\}
\end{align}
admits at least one solution. Moreover, $\varphi\in W^{1,q}(\Omega),\ q\geq1$. Since we do not know whether $\det \nabla \varphi\geq c>0$ and since the energy does not satisfy a polynomial growth condition, it is not clear whether the Euler-Lagrange equations are satisfied in a weak sense.
\end{theorem}

\medskip

\noindent\textbf{Acknowledgment. }  Miroslav  {\v{S}}ilhav{\'y} was supported by grant RVO: 67985840.

\bibliographystyle{plain} 
\addcontentsline{toc}{section}{References}

\begin{footnotesize}

\appendix
\section*{Appendix}\addcontentsline{toc}{section}{Appendix} \addtocontents{toc}{\protect\setcounter{tocdepth}{-1}}

\setcounter{section}{1}

\subsection{The convexity of the largest singular value of $F$}\label{convmax}
\setcounter{subsection}{1}
\setcounter{equation}{0}
\setcounter{theorem}{0}

In this section we outline the proof of the convexity of the function $F\mapsto \lambda_{\rm max}(F)$ (the largest singular value of $F$). First,  we present three lemmas given in \cite[page 364]{Ball77}.
\begin{lemma}\label{LB1}{\rm (von Neumann \cite{neumann1937some}; see also Mirsky \cite{mirsky1959trace,mirsky1975trace})} Let $A, B\in \R^{n\times n}$ have singular values $\alpha_1\geq \alpha_2\geq ...\geq \alpha_n\geq 0$ and $\beta_1\geq \beta_2\geq ...\geq \beta_n\geq 0$. Then
$
|{\rm tr}(A\, B)|\leq \langle \alpha,\beta\rangle.
$
\end{lemma}

\begin{lemma}\label{LB2}{\rm (von Neumann \cite{neumann1937some}; see also Mirsky \cite{mirsky1959trace,mirsky1975trace})} Under the hypotheses of Lemma \ref{LB1}
$
\max\limits_{Q,R\in{\rm O}(n)}\,|\langle A\, Q,R^T \, B^T\rangle|=\langle \alpha,\beta\rangle.
$
\end{lemma}
\begin{proof}
There exist orthogonal matrices $Q_1,Q_2,R_1,R_2$ such that $A=Q_1\diag(\alpha)\, R_1$, $B=Q_2\diag(\beta)\, R_2$. Choose $Q=R_1^TQ_2^T$, $R=R_2^T Q_1^T$. Then
$
\tr(A\, Q\, B\, R)=\langle \alpha, \beta\rangle.
$
But for any orthogonal $Q, R$ the matrices $A\, Q$ and $B\, R$ have singular values $\alpha$, $\beta$ respectively. Hence $|\langle A\, Q,R^T \, B^T\rangle|=\tr(A\, Q\, B\, R)\leq \langle \alpha,\beta\rangle$ by Lemma \ref{LB1}.
\end{proof}

Using the above two lemmas, Ball \cite[page 364]{Ball77} proved:
\begin{proposition}\label{LB3}{(\rm Ball \cite[page 364]{Ball77}} Let $r_1\geq r_2\geq ...\geq r_n\geq 0$. Then $\langle r, \lambda\rangle $ is a convex function of $F$, where  $\lambda_1\geq \lambda_2\geq ...\geq \lambda_n\geq 0$ are the  singular values of $F$.
\end{proposition}
\begin{proof}
In Lemma \ref{LB2} put $A=F$ and $B=\diag(r_1,r_2,...,r_n)$. Then $\langle r, \lambda\rangle=\max_{Q,R\in{\rm O}(n) }\tr(F\, Q\, B\, R)$. Since each $\tr(F\, Q\, B\, R)$ is a convex function of $F$, we obtain that $\langle r, \lambda\rangle$ is also a convex function of $F$.
\end{proof}

Moreover, it holds true that:
\begin{cor}
The function $F\mapsto \lambda_{\rm max}=\lambda_1=\max_{i=1,...,n}\lambda_i(F)$ is a convex function of $F$, where $\lambda_1\geq \lambda_2\geq ...\geq \lambda_n\geq 0$ are the  singular values of $F$.
\end{cor}
\begin{proof}
By letting $r=(1,0,0,...,0)$ in Lemma \ref{LB3} it follows that for $\lambda_1(F)$ is a convex function of $F$.
\end{proof}

\end{footnotesize}
\end{document}